\def\draft{n}
\documentclass{amsart}
\usepackage[headings]{fullpage}
\usepackage{amssymb,epic,eepic,epsfig,amsmath,amsbsy,amsmath,bbm}
\usepackage{graphicx,cite}
\usepackage{url}
\usepackage[bookmarks=true,%
colorlinks=true,%
linkcolor=blue,%
citecolor=blue,%
filecolor=blue,%
menucolor=blue,%
urlcolor=blue,%
breaklinks=true]{hyperref}
\newtheorem{theorem}{Theorem}[section]
\theoremstyle{definition}
\newtheorem{proposition}[theorem]{Proposition}
\newtheorem{lemma}[theorem]{Lemma}
\newtheorem{definition}[theorem]{Definition}

\newtheorem{corollary}[theorem]{Corollary}
\newtheorem{conjecture}[theorem]{Conjecture}

\newtheorem{problem}[theorem]{Problem}

\def\printname#1{
\if\draft y
\smash{\makebox[0pt]{\hspace{-0.5in}
\raisebox{8pt}{\tt\tiny #1}}}
\fi
}

\newcommand{\psdraw}[2]
{\begin{array}{c} \hspace{-1.3mm}
\raisebox{-4pt}{\epsfig{figure=draws/#1.eps,width=#2}}
\hspace{-1.9mm}\end{array}}

\catcode`\@=11
\long\def\@makecaption#1#2{%
\vskip 10pt
\setbox\@tempboxa\hbox{
\small\sf{\bfcaptionfont #1. }\ignorespaces #2}%
\ifdim \wd\@tempboxa >\captionwidth {%
\rightskip=\@captionmargin\leftskip=\@captionmargin
\unhbox\@tempboxa\par}%
\else
\hbox to\hsize{\hfil\box\@tempboxa\hfil}%
\fi}
\font\bfcaptionfont=cmssbx10 scaled \magstephalf
\newdimen\@captionmargin\@captionmargin=2\parindent
\newdimen\captionwidth\captionwidth=\hsize
\catcode`\@=12
\def\lbl#1{\label{#1}\printname{#1}}
\def\BN{\mathbb N}
\def\BZ{\mathbb Z}

\def\BQ{\mathbb Q}
\def\BR{\mathbb R}
\def\BC{\mathbb C}

\def\D{\Delta}

\def\calI{\mathcal I}

\def\calE{\mathcal{E}}
\def\a{\alpha}

\def\La{\Lambda}
\def\l{\lambda}
\def\Ga{\Gamma}

\def\ga{\gamma}

\def\la{\langle}
\def\ra{\rangle}

\def\e{\epsilon}
\def\Ga{\Gamma}
\def\d{\delta}
\def\b{\beta}
\def\th{\theta}
\def\Th{\Theta}

\def\longto{\longrightarrow}

\def\ft{\mathfrak{t}}
\def\calC{\mathcal{C}}

\def\={\;=\;} 
\def\+{\,+\,}
\begin{document}
\title[Asymptotics of classical spin networks]{
Asymptotics of classical spin networks}
\author{Stavros Garoufalidis}
\address{School of Mathematics \\
Georgia Institute of Technology \\
Atlanta, GA 30332-0160, USA \newline 
{\tt \url{http://www.math.gatech.edu/~stavros}}}
\email{stavros@math.gatech.edu}
\author{Roland van der Veen}
\address{Department of Mathematics \\
University of California, Berkeley \\
Berkeley, CA 94720-3840, USA 
\newline 
{\tt \url{http://www.math.berkeley.edu/~roland}}}
\email{roland@math.berkeley.edu}
\thanks{S.G. was supported in part by NSF. R.V. was supported in part by 
the Netherlands Organisation for Scientific Research (NWO). \\
\newline
1991 {\em Mathematics Classification.} Primary 57N10. Secondary 57M25.
}
\date{December 19, 2011}
\dedicatory{\rm{With an appendix by Don Zagier}}
\begin{abstract}
A spin network is a cubic ribbon graph labeled by representations of 
$\mathrm{SU}(2)$. Spin networks are important in various areas of 
Mathematics (3-dimensional Quantum Topology), 
Physics (Angular Momentum, Classical and Quantum Gravity) 
and Chemistry (Atomic Spectroscopy).
The evaluation of a spin network is an integer number. 
The main results of our paper are: 
(a) an existence theorem for the asymptotics of evaluations of
arbitrary spin networks (using the theory of $G$-functions), 
(b) a rationality property of the generating
series of all evaluations with a fixed underlying graph
(using the combinatorics of the chromatic evaluation of a spin network), 
(c) rigorous effective computations of our results for some $6j$-symbols 
using the Wilf-Zeilberger theory, and (d) a complete analysis of 
the regular Cube $12j$ spin network (including a non-rigorous guess
of its Stokes constants), in the appendix.
\end{abstract}
\maketitle
\tableofcontents

\section{Introduction}
\lbl{sec.intro}

\subsection{Spin networks in mathematics, physics and chemistry}
\lbl{sub.history}
A (classical) {\em spin network} $(\Ga,\ga)$ consists of a {\em cubic ribbon
graph} $\Ga$ (i.e., an abstract trivalent graph with a cyclic ordering of the 
edges at each vertex) and a coloring $\ga$ of its set of edges by natural 
numbers. According to Penrose, spin networks correspond to a diagrammatic 
description of tensors of representations of $\mathrm{SU}(2)$. Here a color 
$k$ on an edge indicates the $k+1$ dimensional irreducible representation of 
$\mathrm{SU}(2)$, and their evaluation is a contraction of the above tensors. 
Spin networks originated in work by Racah and Wigner in atomic spectroscopy 
in the late forties \cite{Ra1,Ra2,Ra3,Ra4,Wi}. Exact or asymptotic evaluations 
of spin network is a useful and interesting topic studied by Ponzano-Regge, 
Biedenharn-Louck and many others; see \cite{BL1,BL2,PR,VMK}.
In the past three decades, spin networks have been used in relation to
classical and quantum gravity and angular momentum 
in 3-dimensions; see \cite{EPR,Pe1,Pe2,RS}. In mathematics, 
$q$-deformations of spin networks (so called quantum
spin networks) appeared in the eighties in the work of 
Kirillov-Reshetikhin \cite{KR}. {\em Quantum spin networks} are {\em knotted
framed trivalent graphs} embedded in 3-space with a cyclic ordering of 
the edges near every vertex, and their evaluations are rational functions
of a variable $q$. The quantum theta and $6j$-symbols 
are the building blocks for topological invariants of closed
3-manifolds in the work of Turaev-Viro \cite{TV,Tu}. Quantum spin networks
are closely related to a famous invariant of knotted 3-dimensional objects,
the celebrated {\em Jones polynomial}, \cite{J}. A thorough discussion of 
quantum spin networks and their relation to the Jones polynomial and the 
{\em Kauffman bracket} is given in \cite{KL} and \cite{CFS}.
Recent papers on asymptotics of spin networks in physics and mathematics 
include: \cite{AHHJLY10}, \cite{LY11} and \cite{CM}.
Aside from the appearances of spin networks in the above mentioned areas,
their evaluations and their asymptotics lead to challenging questions
even for simple networks such as the cube, discussed in detail in
the appendix. Some examples of spin networks that will be discussed in the 
paper are shown in Figure \ref{fig.exspin}.
\begin{figure}[htpb]
$$
\psdraw{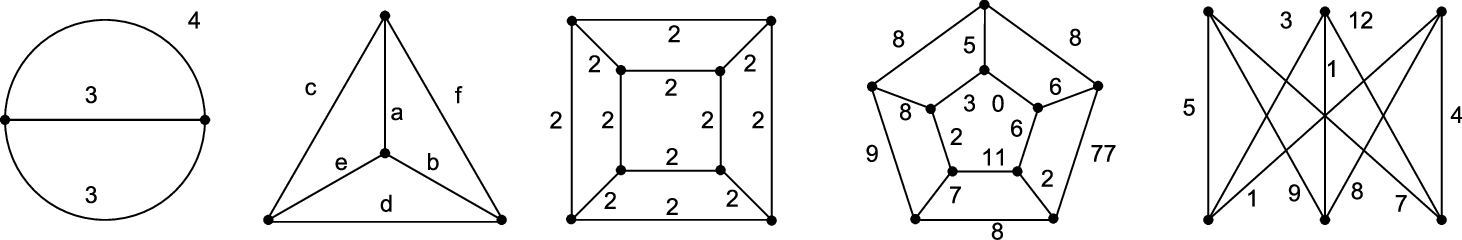}{6in}
$$
\caption{From left to right: The theta, the tetrahedron or 
$6j$-symbol, the Cube, the $5$-sided prism and the complete bipartite graph 
$K_{3,3}$ or $9j$-symbol. The cyclic order of the edges around each vertex
is counterclockwise. The left three spin networks are admissible, and
the right two are not.}\lbl{fig.exspin}
\end{figure}

\subsection{The evaluation of a spin network}
\lbl{sub.standardev}
\begin{definition}
\lbl{def.evalP}
\rm{(a)} We say a spin network is {\em admissible} when the sum of the 
three labels $a,b,c$ around every vertex is even and $a,b,c$ satisfy the 
triangle inequalities: $|a-b|\leq c\leq a+b$. 
\newline
\rm{(b)}
The Penrose evaluation $\la \Ga,\ga\ra^P$ of a spin network $(\Ga,\ga)$
is defined to be zero if it is not admissible. If it is admissible,
its evaluation is given by the following algorithm. 
\begin{itemize}
\item
Use the cyclic ordering to thicken the vertices into disks and the 
edges into untwisted bands. 
\item
Replace the vertices and edges by the linear combinations of arcs as 
follows:
\begin{equation}
\lbl{fig.evaluation}
\psdraw{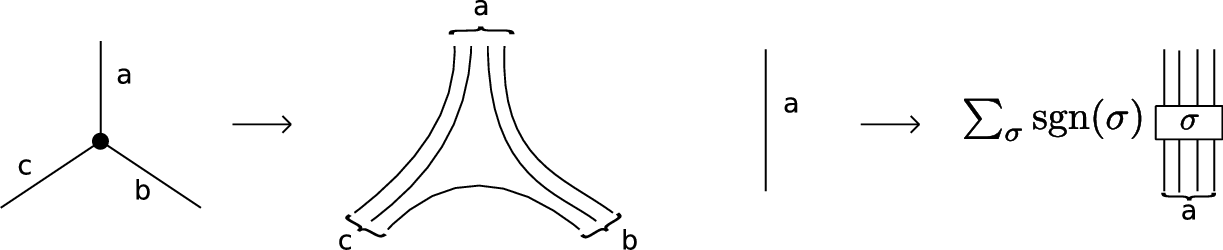}{5in}
\end{equation}
\item
Finally 
the resulting linear combination of closed loops is evaluated by assigning 
the value $(-2)^n$ to a term containing $n$ loops. 
\end{itemize}
\end{definition}
In the above definition the summation is over all permutations $\sigma$ of 
the $a$ arcs at an edge colored $a$. The Penrose evaluation $\la \Ga,\ga\ra^P$ 
of a spin network is always an integer.
Note that the admissibility condition is equivalent to saying that the 
strands can be connected at each vertex as in Figure \ref{fig.evaluation}. 
Note also that cubic ribbon graphs $\Ga$ are allowed to have multiple edges, 
loops and several connected components including 
components that contain no vertices. In addition, $\Ga$ is allowed to
be non-planar (contrary to the requirement of many authors \cite{We,Mou,KL}),
as long as one fixes a cyclic ordering of the edges at each vertex. The latter
condition is implicit in \cite{Pe1}. It turns out that changing the cyclic
ordering at a vertex of a spin network changes its evaluation by single sign;
see Lemma \ref{lem.evaluation} below.

\subsection{Three fundamental problems}
\lbl{sub.problem}
It is easy to see that if $(\Ga,\ga)$ is an admissible spin network and 
$n$ is a natural number, then $(\Ga,n \ga)$ is also admissible. A 
fundamental problem is to study the asymptotic behavior of the sequence of 
evaluations $\la\Ga,n\ga\ra^P$ when $n$ is large. This problem actually 
consists of separate parts. Fix an admissible spin network $(\Ga,\ga)$.
\begin{problem}
\lbl{prob.key1}
Prove the existence of an asymptotic expansion of the sequence
$\la \Ga, n\ga\ra^P $ when $n$ is large.
\end{problem}
\begin{problem}
\lbl{prob.key2}
Compute the asymptotic expansion of the sequence $\la \Ga, n\ga\ra^P $
to all orders in $n$ effectively.
\end{problem}
\begin{problem}
\lbl{prob.key3}
Identify the terms in the asymptotic expansion of $\la \Ga, n\ga\ra^P $
with geometric invariants of the spin network. 
\end{problem}
These problems are motivated by the belief that the quantum 
mechanics of particles with large spin will approximate the classical theory.
To the best of our knowledge, the literature for Problem \ref{prob.key1} 
is relatively new and short and concerns only thetas and $6j$-symbols with 
certain labellings. For Problem \ref{prob.key2}, it should be noted that even 
for the $6j$-symbols not much is known about the subleading terms in the 
asymptotic expansion. Some terms are found in \cite{DL9} but no general 
algorithm is given. As for the geometric interpretation in Problem 
\ref{prob.key3} again there is a well known conjecture in the case of the 
$6j$-symbol \cite{PR}. Roberts used geometric quantization techniques to 
prove this conjecture on the leading asymptotic behavior of $6j$-symbols 
in the so-called Euclidean case \cite{Rb1,Rb2}. 
Some results on the $9j$-symbol have been found in \cite{HL10}. Finally 
more general interpretation for the leading order asymptotics appears in 
\cite{CM} however this assumes a hypothesis that has not been shown to 
hold in cases other than the $6j$-symbol.
Problems \ref{prob.key1}--\ref{prob.key3} can also be viewed as the classical
analogue of the problem of understanding the asymptotics of quantum 
spin networks and quantum invariants. Even less is known in the quantum 
case but see \cite{GV} and a well known conjecture in this context is the 
volume conjecture \cite{Ka},\cite{MM}. 

\subsection{A solution to Problem \ref{prob.key1}}
\lbl{sub.results}
In this paper we give a complete solution to Problem 
\ref{prob.key1} in full generality. A convenient role is played by the 
following normalization of the spin network evaluation. This normalization 
was introduced independently in \cite{C9} in the $q$-deformed case.
\begin{definition}
\lbl{def.eval}
We define the {\em standard normalization} of a spin network evaluation to be 
\begin{equation}
\lbl{eq.evalstandard} 
\la\Ga,\ga\ra = \frac{1}{\calI!}\la\Ga,\ga\ra^P
\end{equation}
where $\mathcal{I!}$ is defined to be the product
\begin{equation}
\lbl{eq.evI}
\mathcal{I}! = 
\prod_{v\in V(\Ga)} 
\left(\frac{-a_v+b_v+c_v}{2}\right)!
\left(\frac{a_v-b_v+c_v}{2}\right)!
\left(\frac{a_v+b_v-c_v}{2}\right)!
\end{equation}
where $a_v, b_v,c_v$ are the colors of the edges emanating from vertex $v$,
and $V(\Ga)$ is the set of vertices of $\Ga$. 
\end{definition} 
The standard normalization has a number of useful properties (see Theorem 
\ref{thm.Eval} below) that can be stated conveniently in terms of a 
generating function that we now define. If we fix a cubic ribbon graph 
$\Ga$ one can consider many spin network evaluations, one for each 
admissible labeling $\ga$ of $\Ga$. We organize these in a generating 
function by taking a formal variable for every edge and encoding $\ga$ 
in the exponents of monomials in these variables.
\begin{definition} 
\lbl{def.Generating}
Given a cubic ribbon graph $\Ga$ define a formal power series in the 
variables $z= (z_e)_{e\in E(\Ga)}$ by
\[F_\Ga(z) = \sum_{\ga\geq 0} \la\Ga,\ga\ra z^\ga \] 
where $z^{\ga}=\prod_{e \in E(\Ga)} z_e^{\ga(e)}$, and $E(\Ga)$ denote the set
of edges of $\Ga$.
\end{definition}
By virtue of our use of the standard normalization we can prove the 
following theorem about our generating function $F_\Ga$. 
\begin{theorem}
\lbl{thm.Eval}
\begin{enumerate}
\item{
For all spin networks $(\Ga,\ga)$, the standard evaluation 
$\la \Ga,\ga \ra$ is an integer.}
\item{
The sequence $\la \Ga, n \ga \ra$ is exponentially bounded.}
\item{ For any cubic ribbon graph $\Ga$ the generating series $F_\Ga$ 
is rational function explicitly defined in terms of $\Ga$.
}
\end{enumerate}
\end{theorem}
To illustrate the last part of the theorem let us mention the special 
case in which $\Ga$ is planar with the counterclockwise orientation.
In this case a result from \cite{We} that states that 
$$
F_{\Ga}(z)= \frac{1}{P_{\Ga}(z)^2}
$$
where $P_\Ga(z) = \sum_{c\in C_\Ga}z^c$ and $C_\Ga$ is the set of
$2$-regular subgraphs of $\Ga$.
Our theorem generalizes this result to arbitrary $\Ga$, the precise 
statement can be found in Theorem \ref{thm.GF}.
The next result gives a complete answer to Problem \ref{prob.key1}.
To state it, we need to recall a useful type of sequence; see 
\cite{Ga4,Gawhat}.
\begin{definition}
\lbl{def.nilsson}
We say that a sequence $(a_n)$ is of 
{\em Nilsson type} if it has an {\em asymptotic expansion} of the form 
\begin{equation}
\lbl{eq.nilsson}
a_n \sim \sum_{\l,\a,\b} \l^{n} n^{\a} (\log n)^{\b} S_{\l,\a,\b}
h_{\l,\a,\b}(1/n)
\end{equation}
where 
\begin{itemize}
\item
the summation is over a finite set of triples $(\l,\a,\b)$, 
\item
the {\em growth rates} $\l$ are algebraic numbers of equal magnitude,
\item
the {\em exponents} $\a$ are rational and the {\em nilpotency exponents} 
$\b$ are natural numbers,
\item
the {\em Stokes constants} $S_{\l,\a,\b}$ are complex numbers,
\item
the $h_{\l,\a,\b}(x)$ are formal power series with coefficients in a number field $K$ such that
the coefficient of $x^n$ is bounded by $C^n n!$ for some $C>0$ and the constant coefficient is $1$.
\end{itemize}
\end{definition}
Note that the set of sequences of Nilsson type is closed under addition
and pointwise multiplication.  Using the theory of $G$-functions, 
(discussed in Section \ref{sub.Gfunction}), we prove:
\begin{theorem}
\lbl{thm.Nilsson}
For any spin network $(\Ga,\ga)$ the sequence $\la \Ga, n\ga \ra$ is of 
Nilsson type. 
\end{theorem}

\subsection{A partial solution to Problem \ref{prob.key2}}
\lbl{sub.results2}
Regarding Problem \ref{prob.key2}, we introduce a new method 
(the Wilf-Zeilberger theory) which 
\begin{itemize}
\item
computes a linear recursion for the sequence $\la \Ga,n \ga\ra$, 
\item
given a linear recursion, effectively computes the corresponding triples 
$(\l,\a,\b)$, the number field $K$ and any number of terms of the power 
series $h_{\l,\a,b}(x) \in 1+xK[[x]]$ in Definition \ref{def.nilsson}, 
\item
numerically computes the Stokes constants $S_{\l,\a,\b}$.
\end{itemize}
Given this information, one may guess exact values of the Stokes constants.
In some cases, we obtain alternative exact computation of the Stokes
constants, too.
As an illustration of the theorem we will present computations of the 
asymptotic expansions of three representative $6j$-symbols up to high order
using the Wilf-Zeilberger method in Section \ref{sub.tet}. In the appendix 
we will present additional numerical results on the asymptotic expansion of 
the case of the cube spin network. About 20 more examples of spin network
evaluations (including the $s$-sided prisms for $s=2,\dots,7$ and the
twisted $s$-sided prisms for $s=2,\dots,5$) have been computed, and the
data is available from the first author upon request. 

\subsection{A conjecture regarding Problem \ref{prob.key3}}
\lbl{sub.results3}
The example of the cube spin network also 
provides evidence for the following conjecture on the growth rates $\l$ in 
the Nilsson type expansion. The conjecture connects the growth rates of 
suitable spin networks to the total mean curvature of a related Euclidean 
polyhedron. Let $P$ be a convex polyhedron in three dimensional Euclidean 
space. Denote by $M(P)$ the total mean curvature of $P$. Recall that 
$M(P) = \frac{1}{2}\sum_{e}\ell_e \phi_e$, where $\phi_e$ is the exterior 
dihedral angle at edge $e$ and $\ell_e$ is the length of the edge. 
\begin{conjecture}
\lbl{conj.Growth}
Let $(\Ga,\ga)$ be a planar spin network such that the dual of $\Ga$ is 
realized as the $1$-skeleton of a convex Euclidean polyhedron $P$ with 
edge lengths given by $\ga$. The numbers $e^{\pm iM(P)}$ are growth rates in 
the asymptotic expansion of the unitary evaluations of $\la\Ga,n\ga\ra^U$.
\end{conjecture}
In the conjecture we are using the so called unitary evaluation 
of a spin network defined in Section \ref{sec.examples}. This evaluation 
differs from the standard one by an explicit factor ensuring that it is 
still of Nilsson type.

After this work was completed, an approach to Problems 
\ref{prob.key1}--\ref{prob.key3} was proposed by Costantino-Marche, 
\cite{CM} using generating functions. Their approach requires certain 
nondegeneracy conditions, and in particular does not give a solution to 
Problem \ref{prob.key1} or Problem \ref{prob.key2} for the regular cube 
spin network, see the Appendix.

\subsection{Acknowledgement}
The results were conceived in a workshop in Aarhus, Denmark, and 
presented in HaNoi, Vietnam and Strasbourg, France in the 
summer of 2007. The authors wish to thank the organizers for their 
hospitality. S.G. wishes to thank C. Koutschan, D. Zeilberger and D. Zagier 
for many enlightening conversations. 

\section{Evaluation of spin networks}
\lbl{sec.eval}
In this section we treat two ways of calculating the evaluation of a spin 
network. The first is by recoupling theory and leads to practical but 
non-canonical formulas for the evaluations as multi-sums. The second way 
is the method of chromatic evaluation. This leads to the proof of the 
generating function result, Theorem \ref{thm.Eval} announced above. 

We start by recording some elementary facts about spin network evaluations. 
First of all our definition of the standard evaluation assumes that there 
are no edges without vertices. By definition we will add an $(a,a,0)$ 
colored vertex to any $a$-colored component that has none. This makes 
sense because of Part (a) of the following.
\begin{lemma}
\lbl{lem.evaluation}
Let $(\Ga, \ga)$ be a spin network and consider the standard evaluation. 
\begin{enumerate}
\item[(a)]{Inserting a vertex colored $(0,a,a)$ in the interior of an edge of 
$\Ga$ colored $a$ does not change the standard evaluation of the spin network.}
\item[(b)]{Changing the cyclic ordering at a vertex whose edges are colored 
$a$, $b$, $c$ changes the evaluation by a sign 
$(-1)^{(a(a-1)+b(b-1)+c(c-1))/2}$}
\end{enumerate}
\end{lemma}
\begin{proof}
(a) The chosen normalization introduces an extra factor $1/a!$ for the new 
vertex labeled $(0,a,a)$, while it follows from the definition that one also 
inserts an extraneous summation over permutations in the pre-existing edge 
labeled $a$. Since 
$$
\sum_{\sigma \in S_a}
\sum_{\tau \in S_a}\mathrm{sign}(\sigma)\sigma\mathrm{sign}(\tau)\tau 
= a!\sum_{\sigma\in S_a}\mathrm{sign}(\sigma)\sigma
$$ 
the evaluation is unchanged.
\newline
(b) Changing the cyclic order at a vertex with edge labels $a,b,c$ has the 
following effect. The alternating sum at each of the adjacent edges is 
multiplied by the 
permutation that turns the arcs in the edge by $180$ degrees. This element has 
sign $a(a-1)/2$ in $S_a$.
\end{proof}
As a consequence of part (a) of the above lemma, an edge labeled $0$ in a 
spin network can be removed without affecting the evaluation. 
There is an alternative {\em bracket normalization} $\la \Ga, \ga \ra^B$
of the evaluation of a spin network $(\Ga,\ga)$ which agrees with a
specialization of the {\em Jones polynomial} or {\em Kauffman bracket}. 
\begin{definition}
\lbl{def.evalB}
The {\em bracket normalization} of a spin network $(\Ga,\ga)$ is defined by
\begin{equation}
\lbl{eq.evalB} 
\la\Ga,\ga\ra^B = \frac{1}{\calE!}\la\Ga,\ga\ra^P
\end{equation}
where 
\begin{equation}
\lbl{eq.evE}
\mathcal{E}! = 
\prod_{e \in E(\Ga)} \ga(e)! 
\end{equation}
This normalization has the property that it coincides 
with the Kauffman bracket (Jones polynomial) of a quantum spin network 
evaluated at $A = -1$ \cite{KL}. However, $\la \Ga, \ga \ra^B$ is not
necessarily an integer number, and the analogous generating series does
not satisfy the rationality property of Theorem \ref{thm.Eval}.
\end{definition}

\subsection{Evaluation of spin networks by recoupling}
\lbl{sub.recoupling}
In this subsection we describe a way of evaluating spin networks by 
recoupling. We will reduce the evaluation of spin networks to 
multi-dimensional sums of $6j$ and theta-symbols. The value of the $6j$ and 
theta-symbols is given by the
following lemma of \cite{KL} and \cite{We}, using our normalization. 
The choice of letters in labeling the $6j$-symbol is traditional following 
for example \cite{KL}.
\begin{lemma}
\lbl{lem.6j3jvalue}
\rm{(a)}
Let $(\psdraw{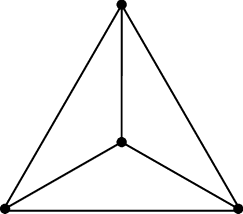}{0.18in},\ga)$ denote a tetrahedron labeled and 
oriented as in Figure \ref{fig.exspin} with 
$\ga=(a,b,c,d,e,f)$. 
Its standard evaluation is given by
\begin{equation}
\lbl{eq.6jsum}
\la \psdraw{tetra}{0.18in}, \ga \ra =
\sum_{k = \max T_i}^{\min S_j} (-1)^{k}(k+1)\binom{k}{
S_1-k , S_2-k , S_3-k , k- T_1 , k- T_2 , k- T_3 , k- T_4}
\end{equation} 
where, as usual
$$
\binom{a}{a_1, a_2, \dots, a_r}=\frac{a!}{a_1! \dots a_r!}
$$
denotes the multinomial coefficient when $a_1+ \dots a_r=a$, and
$S_i$ are the half sums of the labels in the three quadrangular curves in 
the tetrahedron 
and $T_j$ are the half sums of the thee edges emanating from a given vertex. 
In other words, the $S_i$ and $T_j$ are given by
\begin{equation}
\lbl{eq.Si}
S_1 = \frac{1}{2}(a+d+b+c)\qquad S_2 = \frac{1}{2}(a+d+e+f) 
\qquad S_3 = \frac{1}{2}(b+c+e+f)
\end{equation}
\begin{equation}
\lbl{eq.Tj}
T_1 = \frac{1}{2}(a+b+e) \qquad T_2 = \frac{1}{2}(a+c+f)
\qquad T_3 = \frac{1}{2}(c+d+e) \qquad T_4 = \frac{1}{2}(b+d+f).
\end{equation}
\rm{(b)} Let $(\Theta,\ga)$ denote the $\Theta$ spin network of Figure 
\ref{fig.exspin} admissibly colored by $\ga=(a,b,c)$.
Then we have
\begin{equation}
\lbl{eq.3jsum}
\la \Theta, \ga \ra 
= (-1)^{\frac{a+b+c}{2}}\left(\frac{a+b+c}{2}+1\right)
\binom{\frac{a+b+c}{2}}{\frac{-a+b+c}{2}, \frac{a-b+c}{2}, \frac{a+b-c}{2}}.
\end{equation}
\end{lemma}
Finally note that the evaluation of an $n$-labeled unknot is equal to 
$(-1)^n(n+1)$. Recoupling is a way to modify a spin network locally, while 
preserving its evaluation. This is done as in Figure 
\ref{fig.recoupling}. The topmost formula is called the recoupling formula and 
follows from the recoupling 
formula in \cite{KL}, using our conventions. The other two pictures in the 
figure show the 
bubble formula and the triangle formula.
\begin{figure}[htpb]
$$
\psdraw{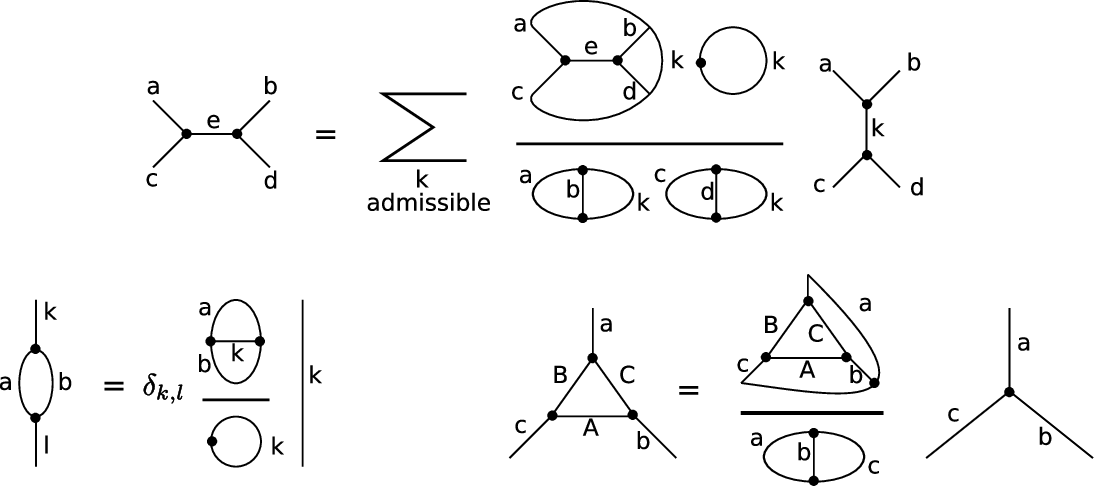}{5in}
$$
\caption{The recoupling formula (top), the bubble formula (left) and the 
triangle formula (right).
The sum is over all $k$ for which the network is admissible, and $\d_{k,l}$
is the {\em Kronecker delta function}.
}\lbl{fig.recoupling}
\end{figure}
The bubble formula shown on the left of Figure \ref{fig.recoupling}
serves to remove all bigon faces. Likewise the triangle formula can be 
used to remove triangles.
To see why the recoupling and bubble formulas suffice to write any spin 
network 
as a multi sum of products of $6j$-symbols divided by thetas we argue as 
follows. Applying the 
recoupling formula to a cycle in the graph reduces its length by one. Keep 
going until you get a multiple edge which can then be removed by the bubble 
formula. 
Although the triangle formula follows quickly from the bubble formula and 
the recoupling formula it is important enough to state on 
its own. For example the triangle formula shows that the evaluation of the 
class of 
\emph{triangular networks} is especially 
simple. The triangular networks are the 
planar graphs that can be obtained 
from the tetrahedron by repeatedly replacing a vertex by a triangle. 
By the triangle formula the evaluation of any triangular network is simply 
a product of 
$6j$-symbols divided by thetas. No extraneous summation will be introduced.
To illustrate how recoupling theory works, let us evaluate the regular 
$s$-sided prism and $K_{3,3}$. Consider the $s$-sided prism network as shown 
in Equation \eqref{eq.drum} (for $s=5$)
where every edge is colored by the integer $n$. In the 
figure we have left out most of the labels $n$ for clarity. By convention 
unlabeled edges are colored by $n$. Performing the recoupling move on every 
inward pointing edge we transform the prism into a string of bubbles that is 
readily evaluated. 
\begin{equation}
\lbl{eq.drum}
\psdraw{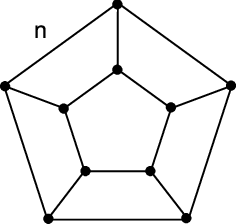}{1in}=\sum_{k \, \text{admissible}}
\left(\psdraw{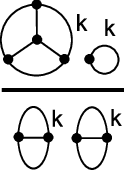}{0.6in}\right)^s \psdraw{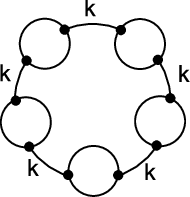}{1in}=
\sum_{k \, \text{admissible}} 
\left(\psdraw{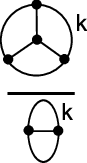}{0.5in}\right)^s \psdraw{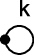}{0.2in}
\end{equation}
Observing that if $n$ is odd the network is not admissible (and thus
evaluates to zero), and 
denoting the tetrahedron and the theta with one edge colored by $k$ and the 
others by $n$ by $S(n,k)$ and $\th(n,k)$ we conclude the following
formula for the $n$-colored $s$-sided prism. 
\begin{proposition}
\lbl{prop.drum}
If $n=2N$ is even we have
$$
\la \mathrm{Prism_s}, 2N \ra=
\sum_{j = 0}^{2N}(2j+1)
\left(\frac{S(2N,2j)}{\th(2N,2j)}\right)^s 
$$
and if $n$ is odd we have 
$\la \mathrm{Prism_s}, n \ra=0$.
\end{proposition}
For small values of $s$ the prism can be evaluated in a more straight forward 
way, thus providing some well known identities of $6j$-symbols. Namely when 
$s = 1$ we get zero, when $s = 2$ we find some thetas and when $s = 3$ we have 
by the triangle formula a product of two $6j$-symbols thus giving a special 
case of the Biedenharn-Elliott identity \cite{KL}. For $s = 4$ we find a 
formula for the regular cube, that will be used in the appendix. 
We know of no easier expression for the evaluation in this case. 
A similar computation for $K_{3,3}$ cyclically ordered as a plane hexagon 
with its three diagonals gives the following.
\begin{proposition}
\lbl{prop.K33}
If $n=2N$ is even we have
$$
\la K_{3,3}, 2N \ra=
\sum_{j = 0}^{2N}(-1)^j(2j+1)\left(\frac{S(2N,2j)}{\th(2N,2j)}\right)^3
$$ 
and if $n$ is odd we have $\la K_{3,3}, n \ra=0$.
\end{proposition}
Note the similarity between $\mathrm{Prism}_3$ and $K_{3,3}$. The only 
difference is the sign that comes up in the calculation when one needs to 
change the cyclic order. The extra sign makes $\la K_{3,3}, 2N\ra = 0$ 
for all odd $N$. 
This is because changing the cyclic ordering at a vertex takes the 
graph into itself, while it produces a sign $(-1)^N$ when all edges are 
colored $2N$.

\subsection{Generating series and chromatic evaluation}
\lbl{sub.chrom}
Recall the generating function $F_\Ga(z)$ for all spin network evaluations 
with the same underlying graph $\Ga$ from Definition \ref{def.Generating}. 
We are using variables $z = (z_e)_{e\in E(\Ga)}$, one for each edge, and 
abbreviate monomials $\prod_{e\in E(\Ga)}z_e^{\ga(e)}$ as $z^\ga$. Our goal is to 
express $F_\Ga$ explicitly in terms of $\Ga$. To do so we need a couple of 
definitions.
\begin{definition}
\lbl{def.cycle}
Given a cubic ribbon graph $\Ga$ define a \emph{cycle} to be a 
(possibly disconnected) $2$-regular subgraph of $\Ga$. The set of all cycles is denoted by 
$C_\Ga$.
\end{definition}
In terms of the cycles we define a polynomial and a quadratic form.
\begin{definition}
\lbl{def.curve}
Given a cubic ribbon graph $\Ga$ and $X\subset C_{\Ga}$ we define
\begin{equation}
\lbl{eq.cycleP}
P_{\Ga,X}(z) = \sum_{c\in C_{\Ga}}\e_X(c) \prod_{e \in c} z_e
\end{equation} 
where $\e_X(c)=-1$ (resp. $1$) when $c \in X$ (resp. $c \not\in X$).
Also define the function $Q_\Ga$ on the subsets of $C_\Ga$ as follows. Let 
$Q_\Ga(X)$ be the number of unordered pairs $\{c,c'\}\subset X$ with the 
property that $c$ and $c'$ intersect in an odd number of places when drawn 
on the thickening of $\Ga$. 
\end{definition}
Note that the cyclic orientation of $\Ga$ defines a unique thickening. Also, 
the polynomials $P_{\Ga,X}$ all have constant coefficient $1$.
We will call $P_\Ga = P_{\Ga,\emptyset}$ the {\em cycle polynomial} of $\Ga$. It is 
independent of the cyclic orientation of $\Ga$. Notice how the other 
$P_{\Ga,X}$ only differ from the cycle polynomial in the signs of the 
individual monomials. It is interesting to note that the cycle polynomial 
determines the cubic graph, up to a well-determined ambiguity. This follows 
from a classic theorem of Whitney, which we quote for the benefit of the 
reader. Recall that a graph $\Ga$ is $2$-connected (resp. $3$-connected) 
if it remains connected after removing any one (resp. any two) vertices of 
$\Ga$.
A {\em Whitney flip} is the following move on a graph (where $R$ and $L$
contain at least two edges):
$$
\psdraw{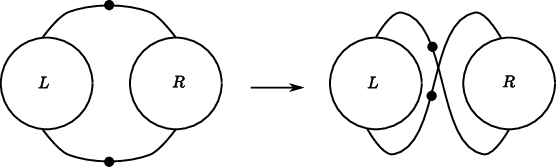}{3in}
$$
A Whitney flip is the graph-theoretic analogue of {\em a knot mutation}
(see \cite{Ad}) and can only be applied to graphs which are not $3$-connected.

\begin{proposition} 
\cite{Wh}
\lbl{prop.Whitney}
\rm{(a)} Let $\Ga_1,\Ga_2$ be two $2$-connected cubic graphs with same
cycle polynomial. Then $\Ga_2$ is obtained from $\Ga_1$ by a sequence
of Whitney flips.
\newline
\rm{(b)} Let $\Ga$ be a $3$-connected cubic graph. The cycle polynomial $P_\Ga$ 
uniquely determines $\Ga$.
\end{proposition}

\begin{proof}
This follows from a more general theorem of Whitney \cite{Wh} that states the following. Two finite
$2$-connected graphs with a bijection on the set of vertices that preserves 
the set of cycles are related by a sequence of Whitney flips. This also works for non-cubic graphs
if we define a {\em cycle} $C$ of a finite graph $\Ga$ to be a subgraph of $\Ga$ with the same vertex set as $\Ga$ such that every
vertex of $C$ has even valency. In case $\Ga$ is a cubic graph, a cycle of $\Ga$ in the above sense exactly
coincides with Definition \ref{def.cycle}. Thus, part (a) follows.

Part (b) follows from (a) and the fact that $3$-connected graphs cannot be
Whitney flipped.
\end{proof}

With the definitions in place we can finally state the precise version 
of the last part of Theorem \ref{thm.Eval}: 
\begin{theorem}
\lbl{thm.GF} 
For every cubic ribbon graph $\Ga$ we have
\begin{equation}
F_{\Ga}(z)=\sum_{X \subset C_\Ga}\frac{a_X}{P_{\Ga,X}^2} \in 
\BZ[[z]] \cap \BQ(z).
\end{equation}
where the coefficients are given by
\[a_X = \frac{1}{2^{|C_\Ga|}}\sum_{Y\subset C_\Ga}(-1)^{Q_\Ga(Y)+|X\cap Y|}\]
\end{theorem}
\begin{corollary}
\lbl{cor.GF}
For every spin network $(\Ga,\ga)$, the evaluation $\la \Ga, \ga \ra$
is an integer number and $\la \Ga, n\ga \ra$ is exponentially bounded. 
\end{corollary}
In particular Theorem \ref{thm.Eval} reduces to Theorem \ref{thm.GF} above.
To see how the particular case of planar spin networks comes about we 
note that
\begin{corollary}
\lbl{cor.We}
When $\Ga$ is planar with the counterclockwise orientation, then all
cycles intersect an even number of times so $(-1)^{Q(X)} = 1$
and hence only $a_{\emptyset}$ is non-zero. It follows
that
$$
F_{\Ga}= \frac{1}{P_{\Ga,\emptyset}^2}
$$
recovering an earlier theorem by Westbury \cite{We}. 
\end{corollary}
The proof of this theorem uses the {\em chromatic evaluation method}
which goes back to \cite{Pe2}. Our proof builds on earlier work by 
\cite{We} and \cite{KL} on planar spin networks and will be given in the next subsection.

\subsection{Chromatic evaluation}
\lbl{sub.chromatic}
\begin{definition}
For $N \in \mathbb{Z}$ define the evaluation $\la \Ga,\ga \ra_N^P$
just as in Definition \ref{def.evalP} except that the value of the a loop 
is now $N$ instead of $-2$.
\end{definition}
Note that by definition $\la \Ga,\ga \ra_{-2}^P = \la \Ga,\ga \ra^P$. However 
for positive $N$ the evaluations are easier to work with. Let $V$ be an 
$N$-dimensional vector space with basis $B = \{b_{1},\hdots b_N\}$. For 
definiteness, let us state that by a loop we mean an immersion of the 
circle. Recall that in the definition 
of the evaluation, the
first step is to replace a diagram of the graph $\Ga$ by a linear 
combination of closed loops. We will call this collection of embedded 
loops the expanded diagram of $\Ga$. Instead of directly evaluating each 
closed loop as $N$ we can also view the expanded diagram as the 
composition (and linear combination) of copies of three types of maps:
\begin{eqnarray*}
\cup: V\otimes V \to \mathbb{C} &\quad \cap: \mathbb{C} \to V \otimes V 
\quad & \times: V\otimes V \to V\otimes V \\
e_i\otimes e_j \mapsto \delta_{i,j} &\quad 1 \mapsto \sum_i e_i\otimes e_i 
\quad & e_i\otimes e_j \mapsto e_j\otimes e_i
\end{eqnarray*} 
Here the direction of composition is upwards. Composing all the maps 
gives a linear map from the ground field $\mathbb{C}$ to itself, which 
is multiplication by the scalar $\la \Ga,\ga \ra_{N}^P$. This interpretation 
of the invariant coupled with a counting argument leads us to an expression 
of the evaluation in terms of cycle configurations that we will define now. 
\begin{definition}
\label{def.loopconf}
Define a \emph{cycle configuration} to be a function $L:C_{\Ga}\to\mathbb{N}$ such that $L(\emptyset)=0$. 
A cycle configuration $L$ defines a coloring $\ga(L)$ as follows 
\[\ga(L)(e) = \sum_{c\in C_\Ga:e\in c}L(c)\quad \text{and}\quad
|L| = \sum_{c\in C_\Ga}L(c)\quad \text{and} \quad L! = \prod_{c\in C_\Ga}L(c)!\] 
Finally define the quadratic form $Q_{\Ga}$ on cycle configurations as 
$Q_\Ga(L) = \sum_{\{c,d\}}L(c)L(d)$, where $\{c,d\}\subset C$ runs over the 
unordered pairs of cycles that intersect in an odd number of places. 
\end{definition}
Viewing a (non-empty) subset $X \subset C_\Ga$ as the cycle configuration that is $1$ 
on $X$ and $0$ elsewhere,
this definition of $Q_\Ga$ coincides with the one given in Definition 
\ref{def.curve}.
We can now state and prove the main lemma that expresses the evaluation 
for positive $N$ in terms of cycle configurations.
\begin{lemma}
\label{lem.evN}
For positive integers $N$ we have
\[
\la \Ga,\ga \ra_N^P 
= \sum_{L:\ga(L) = \ga}(-1)^{Q_\Ga(L)}{
N\choose L}\mathcal{I}! 
\]
\end{lemma}
Here ${N\choose L}$ is defined as $\frac{N(N-1)\hdots(N-|L|+1)}{L!}$ and 
recall $\mathcal{I}!$ is the normalization factor from Definition 
\ref{def.eval}.
\begin{proof}
We view $\la \Ga,\ga \ra_N^P$ as a linear combination of the composition 
of maps consisting of the three elementary maps as explained above. This can 
be made more precise by orienting the edges of $\Ga$ and defining 
$\mathrm{Cups}$ to be the set of the cups (local minima) in the expanded diagram of $\Ga$. 
Now define the product of symmetric groups $S_{\ga} = \prod_{e\in E(\Ga)}S_{\ga(e)}$. 
For a function $f: \mathrm{Cups}\to B$ and $\sigma\in S_\ga$, define 
$\la f, \sigma \ra$ to be $1$ if $f$ is constant along the loops we get 
by connecting the cups using the crossings prescribed by $\sigma$. In all 
other cases $\la f, \sigma \ra = 0$. Composing the elementary maps shows 
directly that the evaluation can be expressed as:
\[\la \Ga,\ga \ra_N^P = \sum_{\sigma \in S_{\ga}}\mathrm{sgn}(\sigma)
\sum_{f:\mathrm{Cups}\to B}\la f,\sigma \ra\]
Call a pair $(f,\sigma)$ good if $\la f,\sigma\ra = 1$, no loops self 
intersect and given any two distinct loops $\ell,\ell'$ sharing an edge 
of $\Ga$ we have $f(\ell) \neq f(\ell')$. Reversing the order of 
summation in the above equation we see that all but the good pairs 
$(f,\sigma)$ cancel out.
A good pair $(f,\sigma)$ determines a unique cycle configuration 
$L_{f,\sigma}$ as follows. Define $L_{f,\sigma}(\emptyset) = 0$ and for non-empty cycles $c$
we define $L_{f,\sigma}(c)$ to be the number of $b\in B$ such that $f^{-1}(b)$ 
traces out $c$. Notice that for any $b\in B$ 
the fiber $f^{-1}(b)$ consists of 
a set of loops that trace out a cycle $c$ without hitting any edge of 
$\Ga$ twice. 

Two good pairs that determine the same cycle configuration must have the 
same sign $\mathrm{sgn}(\sigma)$. To see this we interpret the sign as 
$(-1)^{\#\text{ internal crossings}}$, where a crossing in the expanded diagram is 
\emph{internal} if it comes from a choice of $\sigma$ (not the crossings 
between distinct edges in a projection of $\Ga$). Changing $\sigma$ without 
changing the cycles that are to be traced out will always change the 
number of (internal) crossings by an even amount. It therefore makes sense 
to define the sign of a cycle configuration as $(-1)^{Q(L)}$.
To summarize what we have done so far we can rewrite the evaluation as follows
\[\la \Ga,\ga \ra_N^P = \sum_{L:\ga(L)=\ga} (-1)^{Q(L)} |G'(L)|\]
where $G'(L)$ is the set of all good pairs $(f,\sigma)$ satisfying 
$L_{f,\sigma} = L$. Ordering the basis $B$ and the $|L|$ loops determined by 
$L$ shows that $|G'(L)| = \frac{N!}{(N-|L|)!}|G(L)|$. Here $G(L)$ is the 
subset of $G'(L)$ in which the order of the basis elements agrees with the 
order of the loops. Elements of $G(L)$ are uniquely determined by $\sigma$ 
and $L$ so we will suppress $f$ in the notation. 
We should however keep in mind that using $f$ we are able to distinguish all 
the loops as we follow along the edges of $\Ga$. In particular we know which 
arcs at the vertices of $\Ga$ trace out which cycles.
To finish the proof we need to show that \[|G(L)| = \frac{\mathcal{I}!}{L!}\]
Intuitively this can be understood as the number of ways to arrange the loops
at the vertices, corrected by the symmetry of the cycle configuration. We will however
give a careful proof of this fact below.

We will proceed to count elements of the set $G(L)$ by defining a transitive action of a group of 
order $\mathcal{I}!$ on $G(L)$ with stabilizer of size $L!$. To this end 
call a pair of half edges at a common vertex of $\Ga$ an \emph{angle} and 
denote the set of all angles by $\mathcal{A}$. Looking at the expanded 
diagram of $\Ga$ we define $|a|$ to be the number of arcs passing through 
the angle $a$. By the cyclic orientation on $\Ga$ the angles are 
automatically oriented. Finally define $S_A = \prod_{a\in \mathcal{A}}S_{|a|}$ 
and notice that as promised $|S_A| = \mathcal{I}!$.
The action of $S_A$ on $G(L)$ is defined as follows. At every angle $a$ we 
take a permutation $T_a$ and its inverse and push $T_a$ into the permutation 
at the edge of $\Ga$ it points to and $T_a^{-1}$ into the edge the angle 
points away from. So more formally an element $T = (T_a)_{a\in \mathcal{A}} 
\in S_A$ acts on $\sigma = (\sigma_e)_{e\in E} \in G(L)$ by sending $\sigma_e$ 
to $T_aT_b^{-1}\sigma_eT_cT_d^{-1}$. Here $a,b,c,d$ are the angles adjacent to 
$e$, such that $b$ and $d$ are pointing away from $e$, see Figure 
\ref{fig.Chromatic}. The action is well defined since the pairs $(T_a,T_b)$ 
and $(T_c,T_d)$ commute and the order of multiplication is determined by the 
orientation of the edge $e$. Also the action does not change the cycle 
configuration, just the way the loops are shaped by the permutations so 
it is indeed an action on $G(L)$. 
\begin{figure}[htpb]
$$
\psdraw{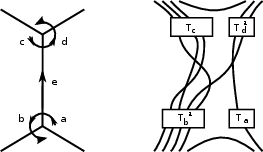}{3.5in}
$$
\caption{On the left we see an edge $e$ of $\Ga$ and its adjacent angles 
$a,b,c,d$ together with their orientations.
In the example on the right the permutation $\sigma_e = (1354)$ is drawn 
in the middle, and the numbers of strands in the four angles are given by 
and $|a| = 1$, $|b| = 4$, $|c| = 3$, $|d| = 2$. The action of $T$ is 
indicated by the white boxes.}\lbl{fig.Chromatic}
\end{figure}
To prove that the action is transitive note that $L$ and $f$ determine how 
the arcs making up the angles are connected. Any $\sigma'\in G(L)$ is 
therefore related to our $\sigma$ by 
$\sigma'_e = \tau_{e,a}\tau_{e,b}\sigma_e\tau_{e,c}\tau_{e,d}$, for some 
$\tau_{e,j} \in S_{|j|}$ permuting the strands at the angle $j$. Let 
$e,e'$ be the edges adjacent to $a$. Since $\sigma$ and $\sigma'$ both 
belong to the same cycle configuration $L$, the product $\tau_{e,a}\tau_{e',a}$ 
can only permute loops that trace the same cycle. The loops do not self 
intersect so the permutations of the loops along any (component of a) cycle 
multiply to $1$. This means that one can modify the permutations $\tau$ so 
that they become inverses along each angle, without changing the product of 
the two permutations adjacent to any edge. Carrying out these modifications 
for each cycle gives the required $T$.
It remains to show that the stabilizer of the action consists of $L!$ 
elements. So suppose we are given a $T \in S_A$ such that at every edge 
$\sigma_e = T_aT_b^{-1}\sigma_eT_cT_d^{-1}$. First of all the $T_j$ can only 
permute arcs belonging to loops in the same cycle. Indeed, the $T_j$ on one 
side of $\sigma_e$ have to be canceled by the $T$'s on the other side. This 
is only possible if the permuted strands connect the same pair of angles at 
the edge. By definition of the action the permutations also have to cancel 
along each angle. This allows us to follow all the non-trivial $T_j$'s 
around, showing that the permuted loops must trace the same cycle.
By the same follow along argument, permuting the arcs in a cycle at a 
single edge determines how they are to be permuted along the whole cycle. 
Therefore the cardinality of the stabilizer is $L!$.
\end{proof}
Since the evaluation $\la \Ga,\ga \ra^P_N$ is a polynomial in $N$, the 
conclusion of the previous
Lemma \ref{lem.evN} actually holds for all $N$, in particular $N = -2$. 
We record this for future use as the following corollary, where we have 
switched back to the standard normalization.
\begin{corollary}
\lbl{cor.lem}
\[
\la \Ga,\ga \ra 
= \sum_{L:\ga(L) = \ga}(-1)^{Q_\Ga(L)}{
-2\choose L}
\]
\end{corollary}

\subsection{Proof of Theorem \ref{thm.GF}}
\lbl{sub.thm.GF}
To find a generating function for these evaluations, we first need to expand 
the sign $(-1)^{Q_\Ga(L)}$ in terms of characters. That is we use Fourier analysis 
on the group $(\mathbb{Z}/2\mathbb{Z})^{|C_\Ga|}$. We often write elements of 
this group as subsets of $C_\Ga$. For every fixed $X\subset C_\Ga$ we have a 
character $(-1)^{X(L)} = (-1)^{\sum_{x\in X}L(x)}$. In case $L$ is a cycle 
configuration we extend the character to a Dirichlet character. So for 
some coefficients $a_X$ we have
\[(-1)^{Q_\Ga(L)} = \sum_{X\subset C_\Ga}a_X (-1)^{X(L)}\]
Taking inner products, the coefficients are given by
\[a_X = \frac{1}{2^{|C_\Ga|}}\sum_{Y\subset C_\Ga}(-1)^{Q_\Ga(Y)+|X\cap Y|}\]
To rewrite the generating function let us introduce variable for each cycle: 
$w = (w_c)_{c\in C_\Ga}$
and set $w_c = \prod_{e\in c}z_e$. If $\ga(L) = L$ then the color of an edge is 
the sum of the number of cycles (with multiplicity) 
containing that edge, hence
\[(-1)^{X(L)}z^{\ga}=(-1)^{X(L)}w^L = \prod_{c\in C_\Ga}(\epsilon_X(c) w_c)^{L(c)}\]
were $\epsilon_X(c)$ is the function that is $1$ if $c\notin X$ and $-1$ if 
$c \in X$.
Recall the cycle polynomial $P_{\Ga,X}(z) =\sum_{c\in C_\Ga} \epsilon_X(c) w_c$ so we 
can compute
\[
\sum_{\ga}\sum_{L:\ga(L) = \ga}{N \choose L}(-1)^{X(L)}z^{\ga} 
= \sum_{L}{N \choose L}\prod_{c\in C_\Ga}(\epsilon_X(c) w_c)^{L(c)} 
= (1+\sum_{\emptyset \neq c\in C_\Ga} \epsilon_X(c) w_c)^N = P_{\Ga,X}^{N}
\]
Now setting $N = -2$ everywhere and applying Corollary \ref{cor.lem} we can 
finish the proof
\[
F_\Ga(z) = \sum_{\ga}\la \Ga,\ga\ra z^{\ga} = \sum_{X\subset C_\Ga}a_X P_{\Ga,X}^{-2}
\]
\qed

\section{Asymptotic expansions}
\lbl{sec.asexp}
The goal of this section is to prove Theorem \ref{thm.Nilsson} that 
provides a Nilsson type asymptotic expansion for spin network evaluations. 
The general idea is to define the (single variable) generating function
\begin{definition}
\lbl{def.singleGen}
Let $(\Ga,\ga)$ be a spin network. The single variable generating function 
$F_{\Ga,\ga}$ is the formal power series
\begin{equation}
\lbl{eq.singleGen}
F_{\Ga,\ga}(z) = \sum_z\la\Ga,n\ga\ra z^n
\end{equation}
\end{definition}
Our goal is to show that this function is a $G$-function (defined in 
Section \ref{sub.Gfunction} below). It then follows from the theory of 
$G$-functions that 
the sequence $\la\Ga,n\ga\ra$ is of Nilsson type. Before doing so we first 
make some comments on Nilsson type sequences in general.

\subsection{Sequences of Nilsson type}
\lbl{sub.nilsson}
Recall from Definition \ref{def.nilsson} that a Nilsson type sequence 
$(a_n)$ has the following asymptotic expansion as $n\to \infty$ 
\[a_n \sim \sum_{\l,\a,\b} \l^{n} n^{\a} (\log n)^{\b} S_{\l,\a,\b}
h_{\l,\a,\b}(1/n)\]
The meaning of this expansion is entirely analogous to the more familiar 
special case where there is only one growth rate:
$$
a_n \sim \l^n n^{\a} (\log n)^{\b} \sum_{k=0}^\infty \frac{\mu_k}{n^k}
$$
(which goes back to Poincar\'e, \cite{O}). In this case the meaning is that 
for every $r \in \BN$ we have
$$
\lim_{n \to \infty} \, n^r\left(a_n \l^{-n} n^{-\a} (\log n)^{-\b}
-\sum_{k=0}^{r-1} \frac{\mu_k}{n^k} \right) =\mu_r.
$$
The general case is similar but to express it we would need more notation 
see \cite{Gawhat}.
It can be shown that a Nilsson type sequence has a unique asymptotic 
expansion \cite{Gawhat}. With respect to point wise addition and 
multiplication the Nilsson type sequences form a $\BC$-algebra. Formally 
multiplying the asymptotic expansions will produce the asymptotic expansion 
of the product sequence. When adding sequences of Nilsson type, the 
asymptotic expansion of the sum sequence is only the sum of the asymptotic 
expansions if all the growth rates have the same absolute value. Otherwise 
only the terms with maximal growth rate occur.
The field $K$ generated by the coefficients of the power series $h_{\l,\a,\b}$ 
in a Nilsson type expansion a number field. See for example Section 
\ref{sec.examples}.
An important source of Nilsson type sequences are $G$-functions that we will 
introduce next.

\subsection{G-functions}
\lbl{sub.Gfunction}
In this section we recall the notion of a $G$-function, 
introduced by Siegel (\cite{Si}) in relation to 
transcendence problems 
in number theory. Many of their arithmetic and algebraic properties were 
established by Andr\'e in \cite{An}. $G$-functions appear naturally in Geometry
(as Variations of Mixed Hodge Structures), in Arithmetic and most recently
in Enumerative Combinatorics. For a detailed discussion, see \cite{An,Ga4}
and references therein. 
\begin{definition}
\lbl{def.Gfunction}
We say that a series $G(z)=\sum_{n=0}^\infty a_n z^n$ is a {\em $G$-function}
if 
\begin{itemize}
\item[(a)]
the coefficients $a_n$ are algebraic numbers,
\item[(b)]
there exists a constant $C>0$ so that for every $n\geq 1$
the absolute value of every conjugate of $a_n$ is less than or equal to 
$C^n$, 
\item[(c)]
the common denominator of the algebraic numbers $a_0,\dots, a_n$ is less 
than or equal 
to $C^n$,
\item[(d)]
$G(z)$ is holonomic, i.e., it satisfies a linear differential equation
with coefficients polynomials in $z$.
\end{itemize}
\end{definition}
For the purposes of this paper the most important property of $G$-functions 
is expressed in the following lemma.
\begin{lemma}\cite[Prop.2.5]{Ga4}
\lbl{lem.Taylor}
The sequence of Taylor coefficients 
of a $G$-function at $z=0$ is a sequence of Nilsson type.
\end{lemma}
With the help of Lemma \ref{lem.Taylor} we can now reduce the proof of 
Theorem \ref{thm.Nilsson} to the following lemma
\begin{lemma}
\lbl{lem.G}
For any admissible spin network $(\Ga,\ga)$ the generating function 
$F_{\Ga,\ga}(z)$ is a $G$-function.
\end{lemma}
In the next subsection we will give a proof of Lemma \ref{lem.G}. 

\subsection{Hypergeometric terms}
In this subsection we prove Lemma \ref{lem.G} (and hence Theorem 
\ref{thm.Nilsson}) by showing that the standard evaluations of spin 
networks are a certain type of hypergeometric multisums that we will first 
describe in general.
\begin{definition}
\lbl{def.hyperg}
An $r$-dimensional {\em balanced hypergeometric datum} $\ft$ 
(in short, {\em balanced datum}) in variables $(n,k)$, where $n \in \BN$ and 
$k=(k_1,\dots,k_r) \in \BN^r$, is 
\begin{itemize}
\item
a finite list $\{(\e_j, A_j(n,k)) \, | j \in J\}$ where 
$A_j: \BN^{r+1} \longto \BZ$
is an affine linear form in $(n,k)$ and $\e_j \in \{-1,1\}$ for all $j \in J$.
\item
a vector $(C_0,\dots,C_r)$ of algebraic numbers
\end{itemize}
that satisfies the {\em balancing condition}
\item
\begin{equation}
\lbl{eq.Ajsum}
\sum_{j=1}^J \e_j A_j=\mathrm{constant}
\end{equation}
and moreover, the set
\begin{equation}
\lbl{eq.Pt}
P_{\ft}=\{x \in \BR^r_{ \geq 0} \, | A_j(1,x) \geq 0 \,\text{for all} \, j \in J\}
\end{equation}
is a {\em compact} rational convex polytope.
\end{definition}
A balanced datum $\ft$ gives rise to a balanced term $\ft(n,k)$
(defined for $n \in \BN$ and $k \in \BZ^r \cap n P_{\ft}$), to 
a sequence $(a_{\ft,n})$ and to a generating series $G_{\ft}(z)$ defined by:
\begin{eqnarray}
\lbl{eq.defterm}
\ft(n,k) &=& C_0^n \prod_{i=1}^r C_i^{k_i} \prod_{j=1}^J A_j(n,k)!^{\e_j}
\\ \lbl{eq.ant}
a_{\ft,n} &=& \sum_{k \in \BZ^r \cap n P_{\ft}} \ft(n,k)
\\ \lbl{eq.Gzt}
G_{\ft}(z) &=& \sum_{n=0}^\infty a_{\ft,n} z^n
\end{eqnarray}
We will call the sequences $(a_{\ft,n})$ balanced multisums. 
The connection between balanced multisum sequences and their asymptotics
was given in \cite{Ga4} using the theory of $G$-functions. More precisely,
\begin{lemma}
\cite[Thm.2]{Ga4}
If $\ft$ is a balanced datum, then
the corresponding series $G_{\ft}(z)$ is a $G$-function.
\end{lemma}
Using this Lemma we can now easily prove Lemma \ref{lem.G}.
\begin{proof} (of Lemma \ref{lem.G})
Using the recoupling formulae from Section \ref{sub.recoupling} we can write 
$\la\Ga,\ga\ra$ as a multi-dimensional sum of products of $6j$-symbols, 
theta-symbols
and unknots (i.e., $1j$-symbols) with a denominator consisting of 
theta-symbols. It follows from Equations \eqref{eq.6jsum} and \eqref{eq.3jsum}
that the $6j$-symbols (resp. theta-symbols) are balanced 1-dimensional 
(resp. $0$-dimensional) sums, thus
the ratio of the product of the theta-symbols by the product of the 
theta-symbols
is a balanced multi-dimensional sum. The unknots can be written as 
$(-1)^k(k+1)!/k!$ and are therefore balanced as well. It is easy to check 
that admissibility guarantees that the multi-dimensional sum has finite range.
\end{proof}
Beware that the term $\ft(n,k)$ constructed in the above proof is neither 
unique nor canonical in any sense. 

\subsection{Integral representation of spin network evaluations}
\lbl{sub.integral}
In this final subsection we comment on the connection between Lemma 
\ref{lem.G} and 
Theorem \ref{thm.Eval} on the rationality of the multivariate generating 
function.
The idea is that the single variable generating function $F_{\Ga,\ga}$ is a 
diagonal of the multivariate generating function $F_\Ga$, where the diagonal 
is defined as follows.
\begin{definition}
\lbl{def.diagonal}
Given a power series $f(x_1,\dots,x_r) \in \BQ[[x_1,\dots,x_r]]$ and
an exponent $J=(j_1,\dots,j_r) \in \BN_+^r$, we define the $J$-diagonal
of $f$ by
\begin{equation}
\lbl{eq.diagonal}
(\D_J f)(z)=\sum_{n=0}^\infty [x^{n J}](f) z^n \in \BQ[[z]]
\end{equation}
where $[x^{n J}](f)$ denotes the coefficient of $x_1^{n j_1} \dots 
x_r^{n j_r}$ in $f$.
\end{definition}
For every spin network $(\Ga,\ga)$ we have
\begin{equation}
\lbl{eq.diagonal2}
F_{\Ga,\ga}=\D_{\ga} F_\Ga
\end{equation}
Consequently, the $G$-function $F_{\Ga,\ga}(z)$ is the diagonal
of a rational function, and thus it comes from geometry in the following 
sense. 
Fix a power series $f(x_1,\dots,x_r) \in \BQ[[x_1,\dots,x_r]]$ convergent at 
the origin and an exponent $J=(j_1,\dots,j_r) \in \BN_+^r$ and consider the 
diagonal $(\D_J f)(z) \in \BQ[[z]]$ as in Definition \ref{def.diagonal}.
Let $\calC$ denote a small real $r$-dimensional torus around the origin.
Then we have the following.
\begin{lemma}
\lbl{lem.diagonal}
With the above assumptions,
\begin{equation}
\lbl{eq.diagonal3}
(\D_J f)(z)=\frac{1}{(2 \pi i)^r} \int_{\calC}
\frac{f(x_1,\dots,x_r)}{x_1^{j_1} \dots x_r^{j_r} -z} d x_1 \wedge \dots \wedge dx_r.
\end{equation}
\end{lemma}
\begin{proof}
With the notation of Definition \ref{eq.diagonal}, 
an application of Cauchy's theorem gives for every natural number
$n$
$$
[x^{nJ}](f)=\frac{1}{(2 \pi i)^r} 
\int_{\calC}
\frac{f(x_1,\dots,x_r)}{(x_1^{j_1}
\dots x_r^{j_r})^{n+1}} dx_1 \wedge \dots \wedge dx_r.
$$
Multiplying by $z^n$ and summing up for $n$ and interchanging summation and 
integration concludes
the proof.
\end{proof}
If in addition $f(x_1,\dots,x_r)$ is a rational function, then the
singularities of the analytic continuation of the right hand-side of
\eqref{eq.diagonal3} can be analyzed by deforming the integration cycle
$\calC$ and studying the corresponding variation of Mixed Hodge Structure
as in \cite{BK}. Such $G$-functions come from geometry; see \cite{An,BK}.

\section{Examples and a conjecture on growth rates}
\lbl{sec.examples}
In this section we illustrate the result of Theorem \ref{thm.Nilsson} on 
the asymptotic expansions in the case of the $6j$-symbol. We also review 
the well known geometric interpretation of the leading asymptotics in this 
case. Finally we formulate a conjecture on the geometric meaning of the 
growth rates in the asymptotic expansion of more general spin networks.
To discuss the geometric aspects of the asymptotics of spin networks it is 
convenient to introduce one more normalization of spin network evaluations
\begin{definition}
\lbl{def.evalU}
We define the {\em unitary normalization} $\la\Ga,\ga\ra^U$ of a 
spin network evaluation $(\Ga,\ga)$ to be 
$$
\la\Ga,\ga\ra^U = \frac{1}{\Th(\ga)} \la\Ga,\ga\ra
$$
where 
$$
\Th(\ga)=\prod_{v\in V(\Ga)}\sqrt{|\la\Th,a_v,b_v,c_v\ra|}
$$
and $a_v,b_v,c_v$ are the colors of the edges at vertex $v$.
\end{definition}
Since the asymptotics of the normalization factor $\Th(\ga)$ is of Nilsson type 
by Stirling's formula \cite{O}, 
we see that $\la \Ga,n\ga\ra^U$ is still of Nilsson type. 
\subsection{The $6j$-symbol and the tetrahedron}
\lbl{sub.tet}
The special case of the tetrahedral spin network or $6j$-symbol motivates 
much of the questions we asked in the introduction. There is a well known 
interpretation of the leading asymptotics in terms of 
a metric tetrahedron $T$ dual to $\Ga$ such that the length of a (dual) 
edge $e$ is given by $\ga(e)$ \cite{PR}. Provided the $6j$ symbol is 
admissible, such a tetrahedron $T$ can always be found uniquely in either 
$\BR^3, \BR^2$ or Minkowski space $\BR^{2,1}$ \cite[Ch.8]{Bl}. We say the 
$6j$-symbol is Euclidean, Plane or Minkowskian depending on the type of $T$. 
The type is determined by the sign of the Cayley-Menger determinant of $T$. 
Let us be more specific in the Euclidean case. Denote by $\phi_e$ is the 
exterior dihedral angle of $T$ at edge $e$.
\begin{theorem}
\lbl{thm.PR}
Let $(\Ga,\ga)$ be a Euclidean $6j$ symbol. The sequence $\la \Ga,n\ga\ra^U$ is of Nilsson type with growth rates, 
Stokes constants and powers of $n$ and $\log n$ are:
$$
\lambda_\pm = e^{i\sum_j\frac{\ga(e_j)\phi_j}{2}} \qquad S_\pm 
= \frac{e^{i\sum_j \frac{\phi_j}{2}+\frac{i\pi}{4}}}{\sqrt{6\pi\mathrm{Vol(T)}}}
\qquad \alpha = \frac{3}{2}\qquad \beta = 0
$$
\end{theorem}
These formulae have been proven in \cite{Rb1}. By
analytically continuing the Euclidean formula for dihedral angles in terms of edge lengths,
the results can be extended to the Minkowskian case. This will be postponed to a future publication.
The Plane case must be different since the volume vanishes in this case. 
Also any interpretation of the terms in the asymptotic expansion of the 
$6j$-symbol beyond the ones just given is very much an open problem \cite{DL9}.
This warrants a detailed and exact investigation of the asymptotics of 
three representative $6j$-symbols using the Wilf-Zeilberger method.
In this method we compute a recursion for the sequence from which all 
terms in the asymptotic expansion except for the Stokes constants may be 
computed. 

We have chosen the simplest examples of a Euclidean $6j$-symbol, a Plane 
one and a Minkowskian $6j$-symbol. Their colorings are given by
$$
\ga_{\mathrm{Euclidean}}=(2,2,2,2,2,2), \qquad
\ga_{\mathrm{Plane}}=(3,4,4,3,5,5), \qquad
\ga_{\mathrm{Minkowskian}}=(4,4,4,4,6,6). 
$$
Using the unitary evaluation (Definition
\ref{def.evalU}) we thus consider the sequences $(a_n), (b_n)$ and $(c_n)$ 
\begin{eqnarray*}
a_n &:=& \la \psdraw{tetra}{0.18in},n \, \ga_{\mathrm{Euclidean}} \ra^U 
= \frac{n!^6}{(3n+1)!^2}
\sum_{k = 3n}^{4n} \frac{(-1)^k(k+1)!}{(k-3n)!^4(4n-k)!^3} 
\\
b_n &:=& \la \psdraw{tetra}{0.18in},n \, \ga_{\mathrm{Plane}} \ra^U
=
\frac{n!^2(2n)!^2(3n)!^2}{(6n+1)!^2}
\sum_{k = 6n}^{7n}\frac{(-1)^k(k+1)!}{(k-6n)!^4(7n-k)!(8n-k)!(9n-k)!}
\\
c_n &:=& \la \psdraw{tetra}{0.18in},n \, \ga_{\mathrm{Minkowskian}} \ra^U 
= \frac{n!^2(3n)!^4}{(7n+1)!^2}
\sum_{k = 7n}^{8n}\frac{(-1)^k(k+1)!}{(k-7n)!^4(8n-k)!(10n-k)!^2}
\end{eqnarray*}
In what follows we denote by $\det(C)$ the Cayley-Menger determinant and by 
$K$ the field
generated by the coefficients of the power series $h_{\l,\a,\b}$ in the 
asymptotic expansion.
The command 
$$
\psdraw{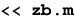}{0.5in}
$$
loads the package of \cite{PaRi} into {\tt Mathematica}. The command
$$
\psdraw{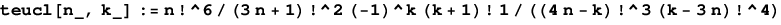}{5in}
$$
defines the summand of the sequence $(a_n)$, and the command
$$
\psdraw{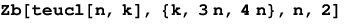}{2in}
$$
computes the following second order linear recursion relation for the
sequence $(a_n)$
$$
\psdraw{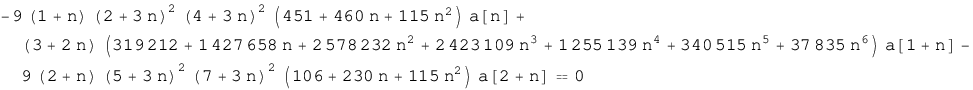}{6in}
$$
This linear recursion has two formal power series solutions of the
form
\begin{eqnarray*}
a_{\pm,n} &=&
\frac{1}{n^{3/2}}\La_{\pm}^n 
\left(1+
\frac{-432 \pm 31 i \sqrt{2}}{576 n}+
\frac{109847 \mp 22320 i \sqrt{2}}{331776 n^2}+
\frac{-18649008 \pm 4914305 i \sqrt{2}}{573308928 n^3} 
\right. \\
& + & \left.
\frac{14721750481 \pm 45578388960 i \sqrt{2}}{660451885056 n^4}+
\frac{-83614134803760 \pm 7532932167923 i \sqrt{2}}{380420285792256 n^5}
\right. \\
& + & \left.
\frac{-31784729861796581 \mp 212040612888146640 i \sqrt{2}}{
657366253849018368 n^6}+
O\left(\frac{1}{n^7}\right)
\right)
\end{eqnarray*}
where
$$
\La_{\pm}=\frac{329 \mp 460 i \sqrt{2}}{729}=e^{\mp i 6 \arccos(1/3)}
$$
are two complex numbers of absolute value $1$. Notice the growth rates 
indeed match the interpretation in terms of dihedral angles of the regular 
tetrahedron predicted in Theorem \ref{thm.PR}.
The coefficients of the formal power series $a_{\pm,n}$ are in the number
field $K=\BQ(\sqrt{-2})$ and the Cayley-Menger determinant is $\det(C)=2^5$.
More is actually true. Namely, the sequence $(a_n)$ generates two 
new sequences $(\mu_{+,n})$ and $(\mu_{-,n})$ defined by
$$
a_{\pm,n}=\frac{1}{n^{3/2}}\La_{\pm}^n \sum_{l=0}^\infty 
\frac{\mu_{\pm ,l}}{n^l}
$$
where $\mu_{\pm,0}=1$. Each of the sequences $(\mu_{\pm,n})$ are factorially
divergent. However, the generating series $\sum_{n=0}^\infty z^n 
\mu_{\pm,n+1}/n!$ are $G$-functions (as follows from \cite{An}), 
and the sequences $(\mu_{\pm,n+1}/n!)$ are of Nilsson
type, with exponential growth rates $\La_{\pm}-\La_{\mp}$. The asymptotics of
each sequence $(\mu_{\pm,n+1}/n!)$ gives rise to finitely many new sequences,
and so on. All those sequences span a finite dimensional vector space, 
canonically attached to the sequence $(a_n)$. This is an instance of 
resurgence, and is explained in detail in \cite[Sec.4]{GM}.
The second order recursion relation for the Plane and the Minkowskian examples
has lengthy coefficients, 
and leads to the following sequences $(b_{\pm,n})$ and 
$(c_{\pm,n})$
\begin{eqnarray*}
b_{+,n} &=&
\frac{1}{n^{4/3}}\La_+^n \left(
1-\frac{1}{3 n}
+\frac{3713}{46656 n^2}
-\frac{25427}{2239488 n^3}
+\frac{9063361}{17414258688 n^4}
-\frac{109895165}{104485552128 n^5}
\right. \\
& & \left. \qquad 
+\frac{1927530983327}{2437438960041984 n^6} + \dots 
\right)
\\
b_{-,n} &=&
\frac{1}{n^{5/3}}\La_-^n \left(
1
-\frac{37}{96 n}
+\frac{3883}{46656 n^2}
-\frac{13129}{4478976 n^3}
-\frac{5700973}{8707129344 n^4}
-\frac{14855978561}{3343537668096 n^5}
\right. \\
& & \left. \qquad
+\frac{2862335448661}{2437438960041984 n^6} + \dots
\right)
\\
c_{\pm,n} &=&
\frac{1}{n^{3/2}} \La_{\pm}^n \left(
1
+\frac{336 \mp 1369 \sqrt{2}}{4032 n}
+\frac{1769489 \mp 831792 \sqrt{2}}{1806336 n^2}
+\frac{67925105712 \mp 66827896993 \sqrt{2}}{21849440256 n^3}
\right. \\
& & \left. \qquad
+\frac{5075437500833257 \mp 2589265090380768 \sqrt{2}}{176193886224384 n^4}
\right. \\
& & \left. \qquad
+\frac{100978405759997442992 \mp 98904713360431641651 \sqrt{2}}{
552544027199668224 n^5}
\right. \\
& & \left. \qquad
+\frac{685103512739058526758457 \mp 349782631602887151717776 \sqrt{2}}{
247539724185451364352 n^6} + \dots
\right)
\end{eqnarray*}
where in the Plane case we have
$$
\La_-=\La_+=-1, \qquad K=\BQ, \qquad \det(C)=0 
$$
and in the Minkowskian case we have
\begin{eqnarray*}
\La_+ &=& 
\frac{696321931873-111529584108 \sqrt{2}}{678223072849}=0.794127\dots \\
\La_- &=& \frac{696321931873+111529584108 \sqrt{2}}{678223072849}=1.25924 \dots
\\
K &=&\BQ(\sqrt{2}), \qquad \det(C)=-2^5 3^4.
\end{eqnarray*}
Again as in the Euclidean case the growth rates may be interpreted in 
terms of dihedral angles.

\subsection{A conjecture on growth rates}
\lbl{sub.conj}
In the special case where $(\Ga,\ga)$ is an admissible tetrahedron, we have 
seen a geometric interpretation for the growth rates of $\la \Ga,n\ga\ra^U$. 
We can reformulate this more concisely using mean curvature. Recall that 
for a convex Euclidean polyhedron $P$ in $\BR^3$ the mean curvature is 
defined by $M(P) =\frac{1}{2}\sum_e \ell_e \phi_e$, where $\phi_e$ is the 
exterior dihedral angle at edge $e$ and $\ell_e$ is the length of the edge.
So in the case of the tetrahedron Theorem \ref{thm.PR} says that if there 
exists a Euclidean tetrahedron $T$ whose $1$-skeleton is dual to $\Ga$ and 
whose edge lengths are given by $\ga$, then the growth rates are given by: 
$\{e^{\pm i M(T)}\}$.
We would like to conjecture that the growth rates of a spin network always 
include a growth rate corresponding to the mean curvature of the dual 
polyhedron whenever these terms make sense. For simplicity we formulate the 
conjecture for planar spin networks only.
\begin{conjecture}
\lbl{conj.mc}
Let $(\Ga,\ga)$ be a planar spin network with the counterclockwise 
orientation. Suppose that $\Ga$ is $3$-connected and that its dual can be 
viewed as the $1$-skeleton of a convex Euclidean polyhedron $P$ whose edge 
lengths are given by $\ga$. The set of growth rates of the Nilsson type 
sequence $\la \Ga,n\ga \ra^U$ contains $e^{\pm iM(P)}$. 
\end{conjecture}
By Cauchy's theorem the dual polyhedron $P$ is determined up to isometry by its $1$-skeleton and its 
edge lengths, i.e. by $(\Ga,\ga)$. This follows from the fact that $P$ has only triangular faces and is convex.

As a first test of the conjecture we show that it behaves well under the 
triangle formula on spin networks defined in Section \ref{sub.recoupling}. In particular this
will verify the conjecture for all triangular networks as defined in Section \ref{sub.recoupling}.
Let $(\Ga,\ga)$ and $(\Ga',\ga')$ be two spin networks that both satisfy 
the hypotheses of Conjecture \ref{conj.mc} and denote their dual polyhedra 
by $P$ and $P'$. Furthermore, suppose that $(\Ga',\ga')$ is obtained from 
$(\Ga,\ga)$ by replacing a vertex $v\in \Ga$ by a triangle. Dually this 
implies that $P$ can be produced by cutting off a tetrahedron from $P'$.
\begin{lemma}
If Conjecture \ref{conj.mc} is true for $(\Ga,\ga)$ then it is also true 
for $(\Ga',\ga')$.
\end{lemma}
\begin{proof}
Let the labels around the vertex $v$ be $a,b,c$ and call the labels of the 
edges of new triangle $A,B,C$ as in Figure \ref{fig.recoupling} 
(lower right) and denote by $(\psdraw{tetra}{0.18in},\psi)$ the tetrahedron 
spin network with labels $a,b,c,A,B,C$ that shows up in the triangle formula.
This formula shows that 
\[\la\Ga',n\ga'\ra^U = (-1)^{\frac{n(a+b+c)}{2}}\la\psdraw{tetra}{0.18in},n
\psi\ra^U\la\Ga,n\ga\ra^U\]
since the theta only contributes to a sign in the unitary evaluation. We 
already know Conjecture \ref{conj.mc} holds for tetrahedra with Euclidean 
duals, including $(\psdraw{tetra}{0.18in},\psi)$. Let us call the dual 
Euclidean tetrahedron $T$. 
Multiplying the asymptotic expansions on the right hand side we see that 
the growth rates will include 
\[(-1)^{\frac{(a+b+c)}{2}}e^{i(M(P)+M(T))} = e^{iM(P')}\]
To see why the equality holds note that we can dissect $P'$ into $P$ and 
$T$ along the triangle with labels $a,b,c$ that is dual to the vertex $v$. 
The minus sign coming from the theta accounts for the fact that we are 
working with exterior dihedral angles and these add an additional factor 
of $\pi$ when comparing the angles in a dissection. 
\end{proof}
The Euclidean volume also appears in the asymptotic expansion of the 
tetrahedral spin network, as part of the Stokes constants, see Section 
\ref{sub.tet}. However this does not generalize well to larger networks 
since the volumes do not add under the triangle formula.
In the appendix we will see a less trivial confirmation of the above 
conjecture for the cube spin network.

\section{Challenges and future directions}
\lbl{sub.challenge}
In this section we list some challenges and future directions.
Our first problem concerns a bound on the unitary evaluations.
\begin{problem}
\lbl{prob.1}
Show that the unitary evaluation of a spin network $(\Ga,\ga)$ satisfies
$$
|\la \Ga, \ga \ra^U| \leq 1
$$
\end{problem}
This problem may be solved using unitarity and locality in a way similar
to the proof that the Reshetikhin-Turaev invariants of a closed 3-manifold
grow at most polynomially with respect to the level; see \cite[Thm.2.2]{Ga5}.
Our next problem is a version of the Volume Conjecture for classical spin 
networks with all
edges colored by $2$.
Problem \ref{prob.1} also suggests that the growth rates must be $\leq 1$. 
In case $\ga = 2n$ more seems to be true.
\begin{problem}
\lbl{prob.2}
The growth rates of the sequence $\la \Ga, 2n \ra^U$ are on the unit circle.
\end{problem}
A positive solution to this problem is known for the following
ribbon graphs: the $\Theta$, the tetrahedron, the 3-faced prism and more 
generally for the infinite family of drums; see \cite{Abde}.
More generally one may pose the following
\begin{problem}
\lbl{prob.3}
Give a geometric meaning to the set of growth rates of a spin network.
\end{problem} 
We have formulated Conjecture \ref{conj.mc} as a partial answer to this 
question but 
that concerns only a single special growth rate among many.
Along the same lines one may ask for an interpretation of the rest of the 
asymptotic expansion.
Looking at the case of the $6j$-symbol it seems reasonable to consider the 
number field $K_{\Ga,\ga}$ generated
by the coefficients of the power series $h_{\l,\a,\b}$ in the Nilsson type asymptotic 
expansion of $\la \Ga,n\ga\ra^U$. 
\begin{problem}
\lbl{prob.5}
Give a geometric interpretation of the number field $K_{\Ga,\ga}$ of a 
spin network $(\Ga,\ga)$.
\end{problem}
Also the Stokes constant might have a geometric meaning as it does in the 
case of the tetrahedron.
\begin{problem}
\lbl{prob.6}
Give a geometric meaning to the Stokes constants of the sequence
$\la\Ga,n\ga\ra^U$. 
\end{problem}
The next problem is a computational challenge to all the known asymptotic
methods, and shows their practical limitations.
\begin{problem}
\lbl{prob.4}
Compute the asymptotics of the evaluation $\la K_{3,3}, 2n\ra$ (given
explicitly in Proposition \ref{prop.K33}) and $\la \text{Cube}, 2n\ra$. 
\end{problem}
The next problem is formulated by looking at the examples from Section 
\ref{sub.tet}. 
\begin{problem}
\lbl{prob.7}
Prove that for every coloring $\ga$ of the tetrahedron spin network
$(\psdraw{tetra}{0.18in},\ga)$, the sequence 
$\la \psdraw{tetra}{0.18in},n \ga \ra$ satisfies a second order recursion
relation with coefficients polynomials in $n$. Can you compute the
coefficients of this recursion from $\ga$ alone?
\end{problem}
Let us end this section with a remark. The main results of our paper can be 
extended to evaluations of spin networks corresponding to higher rank 
Lie groups. This will be discussed in a later publication.
\appendix

\section{Asymptotics of the regular cube}
\lbl{sec.app1}

\centerline{\it by Don Zagier} \bigskip

We give the asymptotic expansion of the standard evaluation $a_n$ of the 
1-skeleton of the 3-dimensional cube,
with all edges colored by $2n$. Proposition \ref{prop.drum} implies that 
$(a_n)$ is given by
\begin{equation} 
\lbl{eq.ancube} 
a_n=\sum_{k=0}^{2n}\,(2k+1)\,a_{n,k}^4 
\end{equation}
with
$$ 
a_{n,k} = \sum_j (-1)^j \binom{k}{j-3n}^2 \binom{2n-k}{4n-j} 
\binom{j+1}{2n+k+1}\,,
$$
making it clear that the numbers $(a_n)$ are integral and positive. The first 
few values of $a_n$ are given by
$$ 
\begin{array}{l} a_0 \= 1\,, \\ a_1 \= 6144\,, \\ a_2 \= 505197000\,, 
\\ a_3 \= 77414400000000\,, \\ a_4 \= 13937620296600000000\,, \\
a_5 \= 3685480142898164744060928\,, 
\\ a_6 \= 1038107879077276408534853271552\,, 
\\ a_7 \= 297223547548257752224492840550400000\,, 
\\ a_8 \= 104193297934159421485149830847575156250000\,, 
\\ a_9 \= 35577316035253000096415678610598379040000000000\,, \\
a_{10} \= 12357485751601160513255660198337121351402277161410240\,. 
\end{array} 
$$
\smallskip
We first look for a recursion of the form
\begin{equation} 
\lbl{eq.guessrecan} 
\sum_{j=0}^J c_j(n)\,a_{n+j}\=0 
\end{equation}
with $J$ not too large and $c_j(n)$ polynomials of $n$ of some not too 
large degree~$d$. Using the first few hundred values of $(a_n)$,
we find experimentally a recursion of this form with $J=4$, $d=61$ and 
with $c_j(n)$ given by
$$ 
\begin{array}{l} 
c_0(n) \= \phantom{-} 3^{16} \, (2n + 7) \, (3n + 2)^8 \, (3n + 4)^8 \, (3n + 5)^7 \, (3n + 7)^7 \, (3n + 8) \, (3n + 10) \, (4n + 3) \, (4n + 5) \, P_0(n+3)\,, 
\\
c_1(n) \= -\, 2 \cdot 3^8 \, (n + 1)^5 \, (2n + 3) \, (2n + 7) \, (3n + 5)^7 \, (3n + 7)^7 \, (3n + 8) \, (3n + 10) \, P_1(n)\,, 
\\
c_2(n) \= -\, 2 \cdot 3^4 \, (n + 1)^5 \, (n+2)^7 \, (2n + 5) \, (3n + 8) \,(3n + 10) \, P_2(n)\,, 
\\
c_3(n) \= -\,2 \cdot (n + 1)^5 \, (n + 2)^7 \, (n + 3)^9 \, (2n + 3) \, (2n + 7) \, P_1(-n-5)\,, 
\\
c_4(n) \= \phantom{-} (n + 1)^5 \, (n + 2)^7 \, (n + 3)^9 \, (n + 4)^{11} \, (2n + 3) \, (4n + 15) \, (4n + 17) \, P_0(n+2)\,, 
\end{array} 
$$
where $P_0$, $P_1$ and $P_2$ are irreducible polynomials (normalized to have 
integral coefficients with no common factor) with leading terms
\begin{eqnarray*}
P_0(n) &=& 2^{11}\,3^7 \, 5^5 \, 7 \cdot 23^5 \, 47 \, 
\bigl(n^{26} + \,\text O\bigl(n^{24}\bigr)\bigr)\,, \\
P_1(n) &=& 2^{15} \, 3^7 \, 5^5 \, 7^3 \, 23^5 \, 47^3 \, 
\bigl(n^{38} + 94\,n^{37} + \,\text O\bigl(n^{36}\bigr)\bigr)\,, \\
P_2(n) &=& 2^{15} \, 3^{13} \, 5^5 \, 7 \cdot 19 \cdot 23^5 \, 
47 \cdot 71 \cdot 73 \, \bigl(n^{46} + 115\,n^{45} 
+ \,\text O\bigl(n^{44}\bigr)\bigr) 
\end{eqnarray*}
as $n\to\infty$, and with the polynomial $P_0$ being even. The full 
values are given at the end of the appendix.
\smallskip
To analyze the asymptotics of the solutions of the above recursion, we 
will use the standard Frobenius theory (see e.g.~\cite{Miller,O,Wasow,WZ2}). 
If $C_j$ denotes the top coefficient of the polynomial $c_j(n)$, then we 
find that $\sum_{j=0}^4 C_j\l^j$ factors as 
$(\l - 3^{12})^2\,(\l\,-\,(1+\sqrt{-2})^{24})\,(\l\,-\,(1-\sqrt{-2})^{24})$ 
and that the indicial equation of the root $3^{12}$ has a double root at 
$-9/2$, while the indicial equations of the roots $(1\pm\sqrt{-2})^{24}$ 
both have root $-4$. This implies that $(a_n)$ has an asymptotic expansion
\begin{equation} 
\lbl{eq.cubeexp}
a_n \;\sim\; S_0 \,\frac{3^{12n}}{n^4}\,
\biggl( (\log n + c)\,M_1\Bigl(\frac1n\Bigr) \+ M_2\Bigl(\frac1n\Bigr)\biggr)
\;+\;\Re\biggl( S_1\,
\frac{\bigl(1+\sqrt{-2}\bigr)^{24n}}{n^{9/2}}\,M_3\Bigl(\frac1n\Bigr)\biggr) 
\end{equation}
for some constants $S_0,\,c\in\BR$, $S_1\in\BC$ and power series 
$M_1(x),\;M_2(x)\in\BQ[[x]]$, $\;M_3(x)\in\BQ[\sqrt{-2}][[x]]$,
normalized by requiring that $M_1$ and $M_3$ have constant term~1. 
Notice that the three roots $3^{12}$, $(1\pm\sqrt{-2})^{24}$
have the same absolute value, so that the different terms of this 
expansion all have the same order of magnitude up to powers of~$n$.
Using the acceleration method\footnote{This method is equivalent to the 
Richardson transform, explained in detail in \cite[Sec.5.2]{GIKM}, and 
also in \cite{BO}.}
described at \cite[p.954]{Za1} and in Section 4 of \cite{GMoree}, 
applied to the values of $a_n$ for $n=1000,\dots,1050$, we find
the numerical values of the constants $S_j$ and $c$ and the first few 
coefficients of the power series $M_i$. The former are then recognized as 
\begin{equation} 
\lbl{eq.S123}
S_0\=\frac{3^5}{2^4\pi^6}\,, \qquad c\=\frac74 \log 2\+\log 3\+\ga\,, 
\qquad S_1\=\frac{(1+i)\,(1+\sqrt{-2})^{12}}{2^{31/4}\,\pi^{11/2}},
\end{equation}
and the latter as
\begin{eqnarray*}
M_1(x) &=& 1 \,-\, \frac{14}9\,x \+ \frac{419}{324}\,x^2 \,-\,
\frac{5659}{8748}\,x^3 \+ \frac{84769}{629856}\,x^4 \,\ldots\,, \\
M_2(x) &=& \phantom{0\,-\,}
\quad \frac12\,x \,-\, 
\frac{689}{864}\,x^2 \+ \frac{4771}{7776}\,x^3 \,-\,
\frac{3799441}{22394880}\,x^4 \+ \ldots\,, \\
M_3(x) &=& 1 \,-\, \frac{2080-43\sqrt{-2}}{1152}\,x 
\+ \frac{1985023-114208\sqrt{-2}}{1327104} x^2 \+ \ldots\,. 
\end{eqnarray*}
The acceleration method can give many more terms, but it is easier to 
simply substitute the Ansatz~\eqref{eq.cubeexp}
into the recursion for the $a_n$, thus obtaining as many terms as desired. 
The approximation works very well in practice,
e.g., the maximal relative error between $a_n$ and the right-hand side 
of~\eqref{eq.cubeexp} with 50 terms of the
power series $M_j(1/n)$ is about one part in $10^{105}$ for $n$ between 
$900$ and $1000$. To first order, the above formulas say that
the asymptotics of $a_n$ are given by
$$ a_n \;\sim\; 3^{12n+5}\, \frac{\log(2^{7/4} 3n) \+ \ga 
\+ \text O(1/\sqrt{n})}{\pi^6\,(2n)^4}\;. $$
We end by giving the complete values of the polynomials $P_j(n)$ that 
appear in the recursion relation: 
{\tiny 
\begin{equation*} 
\begin{split}
P_0(n) \=& 29639019676089600000 \,n^{26} - 150687090646682256000 \,n^{24}+306650022810104871540 \,n^{22} \\ 
& - 331831776907297971277 \,n^{20} + 219414205267920364521 \,n^{18} - 96826696589802950226 \,n^{16} \\
& + 29683042452642732342 \,n^{14} - 6233837158945489065 \,n^{12} + 868763697226715493 \,n^{10} - 77173811768742984 \,n^8 \\
& + 4094153904684504 \,n^6 - 111886799053248 \,n^4 + 797085625600\,n^2 + 17508556800 \,,
\end{split} 
\end{equation*} 

\begin{equation*} 
\begin{split}
P_1(n) \= & 51330514060153830297600000 \,n^{38} + 4825068321654460047974400000 \,n^{37} \\ 
& + 219957552931873414824036864000 \,n^{36} + 6478694077195677171946040064000 \,n^{35} \\ 
& + 138596018058877517667573746466240 \,n^{34} + 2295022658488679405177615124025920 \,n^{33} \\
& + 30614929984046498162519595722728508 \,n^{32} + 338072087836667419737764233439922530 \,n^{31} \\
& + 3151590998989517431768295237323718623 \,n^{30} + 25169023605885819585932158912744414906 \,n^{29} \\
& + 174146716308878486922546565722791225448 \,n^{28} + 1053195250756920493731804102357697945572 \,n^{27} \\
& + 5606361750518381240594997946959656401095 \,n^{26} + 26414736794861925209673962053754850002124 \,n^{25} \\
& + 110642699366898526542975775645886257667832 \,n^{24} + 413447345600050228521136991970449404260966 \,n^{23} \\
& + 1381980145537336658418260176761712507602933 \,n^{22} + 4140268295003002172648827386155584658850114 \,n^{21} \\
& + 11132423733718852590472537735822272877436592 \,n^{20} + 26885849594024541421613060268809580068669016 \,n^{19} \\
& + 58334970199614352499186715966601305101299773 \,n^{18} + 113675657006866049496120543251199160823538984 \,n^{17} \\
& + 198774677991170902825182509342222362759066932 \,n^{16} + 311443610656870193242629490576780944439111836 \,n^{15} \\
& + 436334419544283264503767964716530868856380648 \,n^{14} + 545097336381579864877890591441121830311242864 \,n^{13} \\
& + 605044458481431111735014250352650750979996544 \,n^{12} + 594007313574579683774145689072368182659197376 \,n^{11} \\
& + 512886129222805060276163096656047277821413760 \,n^{10} + 386709368743514690018446501443764021880730368 \,n^9 \\
& + 252342374226937131766477472379392715246649344 \,n^8 + 140888462647571785365811030970706748098760704 \,n^7 \\
& + 66312852042204808325857346405562441033596928 \,n^6 + 25796733254687036537088890539848097231134720 \,n^5 \\
& + 8069974595385074631605661061376909102284800 \,n^4 + 1950251347843211319569463651786279088128000 \,n^3 \\
& + 341555150844826683309630400427989401600000 \,n^2 + 38554163497112285346472887524366745600000 \,n \\
& + 2104728968892765569954334578933760000000\,, \phantom{n^9} 
\end{split} 
\end{equation*} 

\begin{equation*} 
\begin{split} 
P_2(n) \= & 34044436942851501889228800000 \,n^{46} + 3915110248427922717261312000000 \,n^{45} \\
& + 219662706791565782848926392832000 \,n^{44} + 8013067972307054904678991211520000 \,n^{43} \\
& + 213693214418431619419298087215485120 \,n^{42} + 4441366173640824819720536004281937600 \,n^{41} \\
& + 74894384718928871397218606262165524844 \,n^{40} + 1053306557874097263334396534371665684400 \,n^{39} \\
& + 12603933341218324699775023159567066766967 \,n^{38} + 130270526199929371052793324163523141276865 \,n^{37} \\
& + 1176678761157600477488177183321806479261195 \,n^{36} + 9375125430920826953262133708501549064882175 \,n^{35} \\
& + 66383537873561799570407105177080643368970738 \,n^{34} + 420307309950545627790271278025374052312931480 \,n^{33} \\
& + 2391581100396727961084015883784259869439400176 \,n^{32} + 12280714697778331715472758947513061657620945580 \,n^{31} \\
& + 57105561410677181109714147899130291409292461050 \,n^{30} + 241147151880727945569092590718565623964677498750 \,n^{29} \\
& + 926919151084739302148924528551196602048329414090 \,n^{28} + 3249130587428846232389342739110375794038903230050 \,n^{27} \\
& + 10401395584522433137223729045847274251204938831280 \,n^{26} + 30443078183785042173235936392545801110737278600700 \,n^{25} \\
& + 81523471295041101121702148739822166982527381249680 \,n^{24} + 199828567976168135158731196334605133971159536918300 \,n^{23} \\
& + 448394647578816808473555576099796146381767002557015 \,n^{22} + 920880328621349858197546061838422939683875830250825 \,n^{21} \\
& + 1730060383317082970163176773077295571528776228823811 \,n^{20} + 2970768265449411986606322605605748766920711781281175 \,n^{19} \\
& + 4657049644272293081174713334539070013552116895070038 \,n^{18} + 6654448413726919814684796853369712557012807014275460 \,n^{17} \\
& + 8650083204510826813891208265741376798744454496660220 \,n^{16} + 10204350936081623227151423121609083529329974289414800 \,n^{15} \\
& + 10892578929911563608834673775127900025242488747892352 \,n^{14} + 10483723584809889429623667272318319527058201296732320 \,n^{13} \\
& + 9059104715506400545606048153659029802430897952414784 \,n^{12} + 6992047378850409545281115023066707881078577900939520 \,n^{11} \\
& + 4790327226399433431900479655636468590589446522365440 \,n^{10} + 2891085435841974142824330005531725765542027596096000 \,n^9 \\
& + 1522663617068026467557018527073299284644425562419200 \,n^8 + 691604038146644417153727841870625107738881053184000 \,n^7 \\
& + 266822516520374901619621632554668192771563340800000 \,n^6 + 85697330246713152918864338753344668302045030400000 \,n^5 \\
& + 22287663074878779648397028175676823501117952000000 \,n^4 + 4507926265942350227509023190826979902899200000000 \,n^3 \\
& + 665045895624929293830598512362684841984000000000 \,n^2 + 63635282511747835599280673505876676608000000000 \,n \\
& + 2962949736504660768760707593463767040000000000\;. \phantom{n^9} 
\end{split} 
\end{equation*} 
}

\section{Further comments on the asymptotics of the regular cube}
\lbl{sec.furthercube}

The guessed recursion relation of the sequence $(a_n)$ from the previous 
section agrees with the result of the independent 
guessing program {\tt Guess} of 
Kauers; \cite{Ka1,Ka2}. The recursion for $(a_n)$ was verified 
for $n=0,\dots,2996$,
where the height (i.e., the number of digits) of $a_{3000}$ is $17162$. 
On the other hand, the coefficients of the polynomials 
$c_k(n)$ are integers with a much smaller height $73$.
In addition, the root $\l_1$ of the characteristic polynomial can be
written in the form:
$$
\l_1=(1+i \sqrt{2})^{24}
=3^{12} e^{12 i \arccos(-1/3)}
$$
where $e^{12 i \arccos(-1/3)}$ is the exponentiated total mean curvature of the
regular Euclidean octahedron (dual to the regular Euclidean cube)
with unit sides. This confirms Conjecture \ref{conj.mc} on the asymptotics of 
evaluations of classical spin networks. The factor $3^12$ comes from the fact that we are considering
the standard normalization and not the unitary one as is done in the Conjecture.

The asymptotic expansion \eqref{eq.cubeexp} is clearly of Nilsson type,
with the presence of logarithms, and Stokes constants which are no longer
algebraic, up to rational powers of $\pi$. This makes it unlikely that stationary phase type methods
will be able to obtain the asymptotic expansion for the regular cube.

\bibliographystyle{hamsalpha}\bibliography{biblio}
\end{document}